\documentclass[12pt]{amsart}
\usepackage{amssymb, amsmath, amsthm, cancel, hyperref, color, enumerate, verbatim, mathtools}
\usepackage{multirow} 
\usepackage{stmaryrd} 

\usepackage[margin=1in, includeheadfoot]{geometry}
\newtheorem{thm}{Theorem}
\newtheorem{rem}{Remark}
\newtheorem{lem}{Lemma}

\newtheorem{prop}{Proposition}

\newcommand{\leg}[2]{\genfrac{(}{)}{}{}{#1}{#2}} 
\newcommand{\Mod}[1]{\ (\mathrm{mod}\ #1)}
\newcommand{\Sk}[2]{S_{#1}(\Gamma_0(#2))}
\newcommand{\ZZ}{\mathbb{Z}}
\newcommand{\NN}{\mathbb{N}}
\newcommand{\QQ}{\mathbb{Q}}
\newcommand{\CC}{\mathbb{C}}

\usepackage{xcolor}

\begin{document}

\title{Cusp forms as $p$-adic limits}

\author{Michael Hanson \and Marie Jameson}

\maketitle

\begin{abstract}
    Ahlgren and Samart relate three cusp forms with complex multiplication to certain weakly holomorphic modular forms using $p$-adic bounds related to their Fourier coefficients. In these three examples, their result strengthens a theorem of Guerzhoy, Kent, and Ono which pairs certain CM forms with weakly holomorphic modular forms via $p$-adic limits. Ahlgren and Samart use only the theory of modular forms and Hecke operators, whereas Guerzhoy, Kent, and Ono use the theory of harmonic Maass forms. Here we extend Ahlgren and Samart's work to all cases where the cusp form space is one-dimensional and has trivial Nebentypus. Along the way, we obtain a duality result relating two families of weakly holomorphic modular forms that arise naturally in each case. 

\end{abstract}

\section{Introduction and statement of results} 

In \cite{Ahlgren-Samart}, S. Ahlgren and D. Samart considered the cusp form with complex multiplication (with $q := \exp(2\pi i z)$)
\[g(z) = \eta^2(4z)\eta^2(8z) = q-2q^5 - 3q^9 + \cdots \in S_2(\Gamma_0(32))\]
as well as the weakly holomorphic modular form
\[F(z) = -g(z)\frac{\eta^6(16z)}{\eta^2(8z)\eta^4(32z)} = \sum_{n\geq -1} C(n)q^n = -\frac{1}{q} + 2q^3 + q^7 + \dots \in M_2^{\infty}(\Gamma_0(32)),\]
whose only pole lies at the cusp $\infty$. Here, $\eta(z)$ is Dedekind's eta-function $\displaystyle \eta(z) := q^{1/24}\prod_{n=1}^\infty(1-q^n).$ Applying a theorem of P. Guerzhoy, Z. Kent, and K. Ono \cite{Guerzhoy-Kent-Ono}, one has that if $p\equiv 3\pmod{4}$ is a prime for which $p\nmid C(p),$ then as a $p$-adic limit, we have that
\[\lim_{m\rightarrow \infty}\frac{F\mid U(p^{2m+1})}{C(p^{2m+1})}=g.\]
However, Ahlgren and Samart were able to provide a strengthened result in this case, proving that for all primes $p\equiv 3\pmod{4}$ and integers $m\geq 0$, we have
\begin{align*}
v_p(C(p^{2m+1})) &= m\\
v_p\left(\frac{F\mid U(p^{2m+1})}{C(p^{2m+1})} - g\right) &\geq m+1,
\end{align*}
where $v_p(\cdot)$ denotes the $p$-adic valuation on $\ZZ\llbracket q\rrbracket.$

Furthermore, Ahlgren and Samart obtained analogous improved results for two other examples involving normalized cusp forms lying in $S_4(\Gamma_0(9))$ and $S_3(\Gamma_0(16),\chi),$ where $\chi$ is the non-trivial Dirichlet character modulo 4. In each of the three examples given in \cite{Ahlgren-Samart}, the relevant space of cusp forms is one-dimensional and the unique normalized cusp form $g$ in that space has complex multiplication.

Ahlgren and Samart mention that their approach would give similar results for a number of other spaces of modular forms. We resolve this claim in Theorem \ref{thm:main1} for all such one-dimensional spaces whose character is trivial.

\begin{thm}\label{thm:main1}
Suppose that $S_k(\Gamma_0(N))$ is one-dimensional and that the unique normalized cusp form
\[g = \sum_{n=1}^\infty a(n)q^n \in S_k(\Gamma_0(N))\]
has complex multiplication. There exists
\[F = \sum_{n=-1}^\infty C(n)q^n \in S_k^\infty(\Gamma_0(N))\]
such that for every odd prime $p$ which is inert in the field of complex multiplication for $g$ and every integer $m\geq 0$ we have that
\begin{align}
    v_p(C(p^{2m+1})) &= (k-1)m \label{Thm1.a} \\
    v_p\left( \frac{F\mid U(p^{2m+1})}{C(p^{2m+1})} - g \right) &\geq (k-1)(m+1). \label{Thm1.b}
\end{align}
\end{thm}

\begin{rem}\label{rem on p=2}
Here, we assume that $p$ is an odd prime since this theorem depends on Proposition \ref{Cong 2}, which may not hold when $p=2.$ However, Proposition \ref{Cong 2} and Theorem \ref{thm:main1} are both true for the space $S_2(\Gamma_0(27)).$ A modified version of Theorem \ref{thm:main1} also holds for $S_4(\Gamma_0(9)),$ as described in \cite[Theorem 4.1]{Ahlgren-Samart}.
\end{rem}

\begin{rem}
Work of R. Dicks \cite{Dicks} also establishes analogous results for weight 2 newforms with complex multiplication that can be expressed as eta-quotients.
\end{rem}

In order to prove this theorem, for each space $S_k(\Gamma_0(N))$ described above we will construct (see Proposition \ref{prop:1}) two families of modular forms with integer coefficients of the form
\begin{align}
    \phi_n &= q^{-n} + \sum_m A_n(m)q^m \in M_{2-k}^\infty(\Gamma_0(N)) \label{eq:phins}\\
    F_m &= -q^{-m} + \sum_n C_m(n)q^n \in S_{k}^\infty(\Gamma_0(N)), \label{eq:Fms}
\end{align}
where the indices $n, m$ are defined for appropriate sets of integers. In particular, if $k=2$ we consider integers $n\geq 2$ and $m\in \{-1\}\cup \NN$, and if $k>2$ then we consider integers $n\geq 2$ and $m\geq -1.$

These families of modular forms also satisfy a beautiful property known as Zagier duality.  For example, for the one-dimensional space $S_2(\Gamma_0(27))$ one can construct the Fourier expansions of a few $\phi_n$ (where, for convenience, we restrict only to the first few $n$ which are congruent to $2\pmod{3}$).
\begin{align*}
\phi_2 &= q^{-2} + q + 2q^{4} - q^{7} + q^{10} - q^{13} + O(q^{15})\\
\phi_5 &= q^{-5} + q + 2q^{4} + 7q^{7} + 8q^{10} - 10q^{13} + O(q^{15})\\
\phi_8 &= q^{-8} + 6q + 5q^{4} + 14q^{7} - 8q^{10} + 30q^{13} + O(q^{15})\\
\phi_{11} &= q^{-11} - 6q - 8q^{4} + 26q^{7} + 44q^{10} + 95q^{13} + O(q^{15})\\
\phi_{14} &= q^{-14} - 7q + 21q^{4} - 27q^{7} + 21q^{10} + 49q^{13} + O(q^{15})
\end{align*}
One can compare these expansions to those of the first few $F_m$ (where, again for convenience, $m$ is congruent to $1\pmod{3}$) to find that the coefficients which appear are exactly the same (although they are arranged differently).
\begin{align*}
F_1 &= -q^{-1} + q^{2} + q^{5} + 6q^{8} - 6q^{11} - 7q^{14} + O(q^{15})\\
F_4 &= -q^{-4} + 2q^{2} + 2q^{5} + 5q^{8} - 8q^{11} + 21q^{14} + O(q^{15})\\
F_7 &= -q^{-7} - q^{2} + 7q^{5} + 14q^{8} + 26q^{11} - 27q^{14} + O(q^{15})\\
F_{10} &= -q^{-10} + q^{2} + 8q^{5} - 8q^{8} + 44q^{11} + 21q^{14} + O(q^{15})\\
F_{13} &= -q^{-13} - q^{2} - 10q^{5} + 30q^{8} + 95q^{11} + 49q^{14} + O(q^{15})
\end{align*}
In fact, the coefficients that appear in the $q$-expansions of these two families are always the same.  In particular, we have the following result.

\begin{thm} \label{thm:main2}
For all integers $n, m$ as described above, we have that \[C_m(n) = A_n(m).\]
\end{thm}

After giving some background in Section \ref{Background}, we show in Section \ref{k, N list} that there are only finitely many one-dimensional cusp form spaces with trivial Nebentypus (Proposition \ref{prop:finitelymany}). In Section \ref{sec:phiandF} we build and study the $\phi_{n}$ and $F_{m}$ necessary for proving Theorem \ref{thm:main1}, and we also prove Theorem \ref{thm:main2}. These families then provide congruences for the coefficients $C(n)$ of $F$ in Section \ref{sec:congs}. Finally, we prove Theorem \ref{thm:main1} in Section \ref{sec:main1}. 

\subsection*{Acknowledgments}
The authors would like to thank Larry Rolen and Ian Wagner for helpful conversations related to this work.


\section{Background} \label{Background}

We review some basic notation and facts from the theory of modular forms (see, for example, \cite{Diamond-Shurman, Ono}). Let $f:\mathbb{H}\rightarrow \CC$ be a function on the upper half-plane. For a matrix $\gamma = \begin{psmallmatrix} a & b \\ c & d \end{psmallmatrix} \in \mathrm{SL}_{2}(\ZZ)$ and integer $k$, let 
\[f(z) \mid_{k} \gamma := (cz + d)^{-k}f(\gamma z).\]
If $N \in \NN$ and $k$ is a positive even integer, let $M_{k}^{!}(\Gamma_{0}(N))$ be the $\CC$-vector space of functions $f$ which are holomorphic on the upper half plane and satisfy $f \mid_{k} \begin{psmallmatrix} a & b \\ c & d \end{psmallmatrix} = f$ for all $\begin{psmallmatrix} a & b \\ c & d \end{psmallmatrix} \in \Gamma_{0}(N).$ Let $M_{k}(\Gamma_{0}(N))$ be the space of such functions which are also holomorphic at all of the cusps of $\Gamma_{0}(N)$. Let $S_{k}(\Gamma_{0}(N)) \subset M_{k}(\Gamma_{0}(N))$ be the subspace consisting of those forms which vanish at all of the cusps. We are particularly interested in the subspace of weakly holomorphic modular forms 
\[M_{k}^{\infty}(\Gamma_{0}(N)) := \{f \in M_{k}^{!}(\Gamma_{0}(N)) : f \text{ is holomorphic at every cusp except possibly } \infty\},\]
and its cuspidal subspace
\[S_{k}^{\infty}(\Gamma_{0}(N)) := \{f \in M_{k}^{!}(\Gamma_{0}(N)) : f \text{ vanishes at every cusp except possibly }\infty\}.\]

Each $f \in M_{k}^{!}(\Gamma_{0}(N))$ has a Fourier expansion at infinity; for $z\in \mathbb{H}$ and $q := \exp(2\pi iz)$ we have \[f(z) = \sum_{n \gg -\infty} a(n)q^{n}\] for some coefficients $a(n)\in\CC$. 

Define $\Theta := \frac{1}{2\pi i}\frac{d}{dz} = q\frac{d}{dq}$. It is well-known that
    \[\Theta^{k-1}: M_{2-k}^{\infty}(\Gamma_{0}(N), \chi) \rightarrow S_{k}^{\infty}(\Gamma_{0}(N), \chi).\]
For every integer $m > 0$, the $U$ and $V$-operators are defined on Fourier expansions by
    \begin{align*}
    &\sum a(n)q^{n} \mid U(m) := \sum a(mn)q^{n} \\
    &\sum a(n)q^{n} \mid V(m) := \sum a(n)q^{mn}.
    \end{align*}
We also let $T_{k}(m)$ be the usual Hecke operator on $M_{k}^{!}(\Gamma_{0}(N))$, which is given by (see, for example, \cite{Diamond-Shurman})
\[\sum a(n)q^n \mid T_{k}(m) = \sum_{n}\left(\sum_{d\mid (m,n)}d^{k-1}a(mn/d^2)\right) q^n.\]
In particular, for $f=\sum a(n)q^n$ we have 
\[f \mid T_{k}(p^{n}) = \sum_{j=0}^{n} p^{(k-1)j} f \mid U(p^{n-j}) \mid V(p^{j}).\]
It is well-known that if $(m, N)=1$ then
\[T_{k}(m) : M_{k}^{\infty}(\Gamma_{0}(N)) \rightarrow M_{k}^{\infty}(\Gamma_{0}(N)),\]
and in fact $T_{k}(m)$ also preserves the cusp form subspace $S^{\infty}_{k}(\Gamma_{0}(N))$.

We restrict our attention to the spaces $S_{k}(\Gamma_{0}(N))$ which are one-dimensional and spanned by a normalized cusp form $g$ having complex multiplication by a quadratic field $K$ with fundamental discriminant $D < 0$. This essentially means that the coefficients of $g$ are supported on those $n \in \ZZ$ for which $\leg{D}{n} = 1$. For a more detailed account of CM forms see, for example, \cite[Section 1.2.2]{Ono} or \cite[Section 5]{Bruinier-Ono-Rhoades}. 

Finally, we will need some facts about filtrations. If $p \nmid 6N$ and $k \geq 0$, let $M_{k}^{(p)}(\Gamma_{0}(N)) \subset M_{k}(\Gamma_{0}(N))$ be the space of forms whose coefficients are $p$-integral rational numbers. If $f \in M_{k}^{(p)}(\Gamma_{0}(N))$, define the filtration
\[w_{p}(f) :=\inf\{k' : f\equiv g \Mod{p} \text{ for some } g \in M_{k'}^{(p)}(\Gamma_{0}(N))\}.\]
The following facts can be found in Section 1 of \cite{Jochnowitz}.

\begin{prop} \label{filt}
If $f \in M_{k}^{(p)}(\Gamma_{0}(N))$ and $w_{p}(f) \neq -\infty$, then $w_{p}(f) \equiv k \Mod{p-1}$. Moreover, $w_{p}(f \mid V(p)) = pw_{p}(f)$.
\end{prop} 

\section{Reducing to a finite list of $(k,N)$} \label{k, N list}

First, we will describe the set of all pairs $(k,N)$ where $S_k(\Gamma_0(N))$ is one-dimensional and spanned by a cusp form with complex multiplication. First, we have that (as in, for example, \cite[Theorem 3.5.1]{Diamond-Shurman})
\[\dim(S_k(\Gamma_0(N))) = \begin{cases} (k-1)(g_N-1) + \lfloor\frac{k}{4}\rfloor \varepsilon_2 + \lfloor\frac{k}{3}\rfloor \varepsilon_3 +  (\frac{k}{2}-1)\varepsilon_\infty & \text{if } k\geq 4\\ g_N & \text{if }k=2\\ 0 & \text{if } k\leq 0\end{cases},\]
where $g_N$ is the genus of $X_0(N),$ $\varepsilon_2$ is the number of elliptic points with period 2, $\varepsilon_3$ is the number of elliptic points with period 3, and $\varepsilon_\infty$ is the number of cusps. So, to find the set of all one-dimensional spaces $S_k(\Gamma_0(N)),$ we need only consider the levels $N$ for which $g_N\leq 1.$

The number of congruence subgroups of a fixed genus is known to be finite, and work of D.A. Cox and W.R. Parry gives explicit bounds on the level (see \cite{Cox-Parry, Cummins-Pauli}), proving that
\[N\leq \begin{cases} 168 & \text{if }g_N=0\\ 12g_{N}+\frac{1}{2}(13\sqrt{48g_N+121})+145) & \text{if }g_N\geq 1\end{cases}.\]
Thus, we need only consider $(k,N)$ where $N \leq 12+\frac{1}{2}(13\sqrt{48+121})+145) = 241.5.$ Finally, after a short Sage calculation, it follows that the only pairs $(k, N)$ such that $S_k(\Gamma_0(N))$ is one-dimensional and spanned by a cusp form with complex multiplication are
\[(2, 27), (2, 32), (2, 36), (2, 49), \text{and } (4, 9).\]
This proves the following proposition.

\begin{prop} \label{prop:finitelymany}
The only one-dimensional spaces $S_k(\Gamma_0(N))$ which are spanned by a cusp form with complex multiplication are
\[\Sk{2}{27}, \Sk{2}{32}, \Sk{2}{36}, \Sk{2}{49}\text{and } \Sk{4}{9}.\]
\end{prop}

\begin{rem}
In particular, there are only five normalized cusp forms $g\in S_k(\Gamma_0(N))$ that have complex multiplication and $\dim(S_k(\Gamma_0(N)))=1.$ Thus, in order to prove Theorem \ref{thm:main1}, it suffices to consider these five cases.  These five cusp forms are listed in Table \ref{tab:prop1}, together with their LMFDB labels and field $K=\QQ(\sqrt{D})$ of complex multiplication (where $D<0$ is the fundamental discriminant of $K$) \cite{LMFDB}.
\end{rem}

\begin{table}[h]
\centering
\begin{tabular}{|c|c|c|c|} \hline
$(k,N)$ & $g = \sum a(n)q^n \in S_k(\Gamma_0(N))$ & LMFDB & $K = \QQ(\sqrt{D})$\\ \hline \hline
$(2,27)$ & $g = \eta(3z)^2\eta(9z)^2 = q - 2q^{4} - q^{7} + \cdots $ & 27.2.a.a & $\QQ(\sqrt{-3})$\\
$(2,32)$ & $g = \eta(4z)^2\eta(8z)^2 = q - 2q^{5} - 3q^{9} + \cdots $ & 32.2.a.a & $\QQ(\sqrt{-4})$\\
$(2,36)$ & $g = \eta(6z)^4 = q - 4q^{7} + \cdots $ & 36.2.a.a & $\QQ(\sqrt{-3})$\\
$(2,49)$ & $g = q + q^{2} - q^{4} - 3q^{8} - 3q^{9} + \cdots$ & 49.2.a.a & $\QQ(\sqrt{-7})$\\
$(4,9)$ & $g = \eta(3z)^8 = q - 8q^{4} + 20q^{7} + \cdots$ & 9.4.a.a & $\QQ(\sqrt{-3})$\\
\hline \end{tabular}
\caption{Cusp forms $g$ with CM that span $S_k(\Gamma_0(N))$}
\label{tab:prop1}
\end{table}

\begin{rem}
Note that S. Ahlgren and D. Samart \cite{Ahlgren-Samart} proved Theorem \ref{thm:main1} for $g\in \Sk{2}{32}$ and $g\in \Sk{4}{9}.$
\end{rem}

\section{Defining $\phi_n$ and $F_m$} \label{sec:phiandF}

First, we recall the following well-known fact (see, for example, the proof of Corollary 2.4 of \cite{El-Guindy-Ono} or Lemma 2.1 of \cite{JenkinsMolnar19}).

\begin{lem}\label{lem:constanttermsvanish}
If $f\in S_2^\infty(\Gamma_0(N))$ then the constant term of $f$ must vanish.
\end{lem}

For each pair $(k,N)$ identified in Proposition \ref{prop:finitelymany}, we can now define the two families of modular forms described in equations \eqref{eq:phins} and \eqref{eq:Fms}.

\begin{prop}\label{prop:1}
For each space $\Sk{k}{N}$ identified in Proposition \ref{prop:finitelymany}, we have the following.
\begin{enumerate}[(a)]
\item For all integers $n\geq 2$ there exists $\phi_n \in M_{2-k}^\infty(\Gamma_0(N))\cap \ZZ(\!(q)\!)$ of the form
\[\phi_n = q^{-n} + A_n(-1)q^{-1} + \sum_{m=1}^\infty A_n(m)q^m\] if $k=2$ and of the form
\[\phi_n = q^{-n} + \sum_{m=-1}^\infty A_n(m)q^m\] if $k=4.$
\item For $m=-1$ and for all integers $m\geq 1$ there exists a unique $F_m \in S_{k}^\infty(\Gamma_0(N))\cap \ZZ(\!(q)\!)$ of the form
\[F_m = -q^{-m} + \sum_{n=2}^\infty C_m(n)q^n.\]
If $k=4,$ then a unique $F_m$ of the given form exists for all integers $m\geq -1.$
\item Define $\displaystyle F := F_{1} = \sum_{n=-1}^\infty C(n)q^n$. If $p$ is prime such that $p \nmid N$, then we have for each $n \geq 0$ that 
\[F \mid T_k(p^{n}) = p^{(k-1)n}F_{p^{n}} + C(p^{n})g.\]
\end{enumerate}
\end{prop}

\begin{proof}
\begin{enumerate}[(a)]
\item For $(k,N)=(2,32),$ note that the modular functions $\phi_{n}\in M_0^\infty(\Gamma_0(32))\cap \ZZ(\!(q)\!)$ are constructed by A. El-Guindy and K. Ono in \cite[Lemma 2.3]{El-Guindy-Ono}. We will construct the $\phi_n$ similarly for $(k,N)\in \{(2,27), (2,36), (2,49)\}.$ First, we explicitly define $\phi_{2}$ and $\phi_{3}$ in those cases as in Table \ref{tab:prop2a}.

\begin{table}[h]
\centering
\begin{tabular}{|c|c|} \hline
$(k,N)$ & $\phi_2, \phi_3\in M_{2-k}^\infty(\Gamma_0(N))$ \\ \hline \hline
\multirow{2}{*}{$(2,27)$} & $\displaystyle\phi_2 := \frac{\eta^{4}(9z)}{\eta(3z)\eta^{3}(27z)} = q^{-2} + q + 2q^{4} - q^7 + \cdots$\\
& $\displaystyle\phi_3 := \frac{\eta^{3}(3z)}{\eta^{3}(27z)} + 3 = q^{-3} + 5q^{6} + \cdots$\\ \hline
\multirow{2}{*}{$(2,32)$} & $\displaystyle\phi_2 := \frac{\eta^6(16z)}{\eta^2(8z)\eta^4(32z)} = q^{-2} + 2q^6 + \cdots$\\
&$\displaystyle \phi_3 := \frac{\eta^4(8z)\eta^2(16z)}{\eta^2(4z)\eta^4(32z)} = q^{-3} + 2q + q^5 + 2q^9+\cdots$\\ \hline
\multirow{2}{*}{$(2,36)$} & $\displaystyle\phi_2 := \frac{\eta(12z)\eta^3(18z)}{\eta(6z)\eta^3(36z)} = q^{-2} + q^4 + \cdots$\\
&$\displaystyle \phi_3 := \frac{\eta^3(9z)\eta(12z)}{\eta(3z)\eta^3(36z)} -1 = q^{-3} + 2q^3 + q^9 + \cdots$\\ \hline
\multirow{2}{*}{$(2,49)$} & $\displaystyle \phi_2 := \frac{\eta(z)}{\eta(49z)} + 1 = q^{-2} - q^{-1} + q^3 + q^5 + \cdots$\\
& $\displaystyle\phi_3 := \phi_2 \mid T_0(2) - \frac{1}{2}\phi_2^2+ \phi_2 = q^{-3} - q + q^2 + q^4 - q^9 + \cdots$\\
\hline \end{tabular}
\caption{Defining $\phi_2, \phi_3\in M_{2-k}^\infty(\Gamma_0(N))$ when $k=2$}
\label{tab:prop2a}
\end{table}
 
Note that each eta-quotient defined here is an element of $M_{0}^\infty(\Gamma_0(N))$ (see, for example, \cite[Theorems 1.64 and 1.65]{Ono}). Then, the modular functions $\phi_{4}, \phi_{5}, \ldots$ can be defined inductively as the appropriate polynomials in $\phi_{2}, \phi_{3}$ to obtain the desired principal parts.

For example, when $(k,N)=(2,49),$ we have
\begin{align*}
\phi_4 &:= \phi_2^2 + 2\phi_3 - \phi_2 = q^{-4} + q^{-1} + q^3 - q^5 + q^6  + \cdots\\
\phi_5 &:= \phi_3\phi_2 + \phi_2^2 + 2\phi_3 - \phi_2 -3 = q^{-5} - q + 2q^2 - q^4 + q^8 + 2q^9 + \cdots\\
\phi_6 &:= \phi_2^3 + 3\phi_3\phi_2 + \phi_3 - 3 = q^{-6} + 2q + q^2 + q^4 + 3q^8 + \cdots,
\end{align*}
etc. This completes the proof of (a) when $k=2.$

If $(k,N)=(4,9),$ note that we cannot define $\phi_n\in M_{-2}^\infty(\Gamma_0(9))\cap \ZZ(\!(q)\!)$ in quite the same way because they are not modular functions.  Instead, we follow the proof of Lemma 4.3 of \cite{Ahlgren-Samart} (with the appropriate modifications to consider all positive integers $n\geq 2$).  Set
\begin{align*}
\phi_2 &:= \frac{\eta^2(3z)}{\eta^6(9z)} = q^{-2} - 2q - q^4 - 8q^7 + \cdots \in M_{-2}^\infty(\Gamma_0(9))\cap \ZZ(\!(q)\!) \\
L &:= \frac{\eta^3(z)}{\eta^3(9z)}+3 = q^{-1} + 5q^2 - 7q^5 + 3q^8 + \cdots \in M_{0}^\infty(\Gamma_0(9))\cap \ZZ(\!(q)\!).
\end{align*}
Then, the modular forms $\phi_3, \phi_4, \ldots$ can be defined inductively as the appropriate integer linear combinations of $\phi_2, \phi_2L, \phi_2L^2, \ldots$ to obtain the desired principal parts. For example, we have that
\begin{align*}
\phi_3 &:= \phi_2L = q^{-3} + 3 - 18q^3 + 20q^6 + 45q^9+ \cdots\\
\phi_4 &:= \phi_2L^2 = q^{-4} + 8q^{-1} - 10q^2 - 88q^5 + 295q^8+ \cdots\\
\phi_5 &:= \phi_2L^3 - 13\phi_2 = q^{-5} + 49q - 178q^4 - 140q^7 + \cdots,
\end{align*}
etc.

\item We first build the $F_{m}\in S_{2}^\infty(\Gamma_0(N))\cap \ZZ(\!(q)\!)$ for all of the weight 2 spaces. In each such case, set $F_{-1} := -g$. Then, to inductively construct the $F_{m}$ for $m \geq 1$ we consider $F_{-1}\phi_{m+1}\in S_{2}^\infty(\Gamma_0(N))\cap \ZZ(\!(q)\!)$. Note that by Lemma \ref{lem:constanttermsvanish} the constant coefficient of $F_{-1}\phi_{m+1}$ is zero. Inductively, we may add the appropriate linear combinations of the previous $F_{r}$, for $r < m$, to get rid of the undesired terms in the principal part of $F_{-1}\phi_{m+1}$. For example, when $(k, N) = (2, 49)$ we have 
    \begin{align*}
        F_{1} &:= F_{-1}\phi_{2} + F_{-1}= -q^{-1} - q^3 - q^5 + 2q^6+\cdots \\
        F_{2} &:= F_{-1}\phi_{3} - F_{1} + F_{-1} = -q^{-2} + q^{3} + 2q^5 + q^6+\cdots \\
        F_{3} &:= F_{-1}\phi_{4} - F_{2} - F_{-1} = -q^{-3} + q^{2} + q^4 - q^8 - 4q^9 +\cdots \\
        F_{4} &:= F_{-1}\phi_{5} - F_{3} + F_{1} = -q^{-4} + q^{3} - q^5 + q^6 +\cdots,
    \end{align*}
etc.

When $(k, N) = (4, 9)$, we again must adjust our argument slightly since the modular forms $\phi_{m}$ have weight $-2$. First, set $F_{-1} := -g.$ Then, the $F_{m}$ for $m\geq 0$ can be constructed inductively by adding the appropriate linear combinations of the previous $F_{r}$, for $r < m$, to $F_{-1}L^{m+1}$.

Uniqueness is proved for all five spaces simultaneously: if $G_{m} = q^{-m} + O(q^{2}) \in S_{k}^{\infty}(\Gamma_{0}(N))$ then $F_{m} - G_{m} = O(q^{2}) \in S_{k}(\Gamma_{0}(N))$. Since the CM form $g = q + O(q^{2})$ spans the one-dimensional space $S_{k}(\Gamma_{0}(N))$, we must have $F_{m} = G_{m}$. 

\item First observe that
\[F \mid T_{k}(p^{n}) = -p^{n(k-1)}q^{-p^{n}} + C(p^{n})g + O(q^{2}).\]
Hence
\[F \mid T_{k}(p^{n}) - p^{n(k-1)}F_{p^{n}} = C(p^{n})g + O(q^{2}) = C(p^{n})g,\]
the last equality following from $\dim S_{k}(\Gamma_{0}(N)) = 1$. This gives the desired identity.

\end{enumerate}
\end{proof}

We now prove Theorem \ref{thm:main2}.
\begin{proof}[Proof of Theorem \ref{thm:main2}]
Note that for $m,n$ as given above, we have that $F_m\phi_n \in S_2^\infty(\Gamma_0(N)).$  By Lemma \ref{lem:constanttermsvanish}, it follows that $F_m\phi_n$ has constant term
\[0 = C_m(n) - A_n(m),\]
as desired.
\end{proof}

\begin{rem}
Note that Theorem \ref{thm:main2} guarantees that the $\phi_n$ are uniquely defined, since the $F_m$ are unique.
\end{rem}

\begin{rem} \label{A(-1) = 0}
By Theorem \ref{thm:main2} we have that
\[A_n(-1) = C_{-1}(n) = a(n).\]
Thus, if $p$ is a prime which is inert in the field of complex multiplication then $A_p(-1)=0.$
\end{rem}

\section{Congruence results for $C(p^{2m+1})$} \label{sec:congs}

\begin{prop} \label{Cong 1}
For each of the spaces $S_{k}(\Gamma_{0}(N))$ listed in Proposition \ref{prop:finitelymany}, if $p$ is an inert prime in the field of complex multiplication for the corresponding newform $g$ and $m\geq 0$ is an integer, we have
\[C(p^{2m+1}) \equiv (-1)^{m}p^{(k-1)m}C(p) \Mod{p^{(k-1)(m+1)}}.\]
\end{prop}

\begin{proof}
By Remark \ref{A(-1) = 0}, if $p$ is an inert prime in the CM field then $\phi_{p}$ has the form $\phi_{p} = q^{-p} + \sum_{n\geq 1} A_{p}(n)q^{n}$. Theorem \ref{thm:main2} implies that $A_{p}(1) = C(p)$. We thus have
\[\Theta^{k-1}(\phi_{p}) = -p^{k-1}q^{-p} + C(p)q + O(q^{2}) \in S_{k}^{\infty}(\Gamma_{0}(N)).\]
On the other hand, part (c) of Proposition \ref{prop:1} gives
        \[F \mid T_{k}(p) = -p^{k-1}q^{-p} + C(p)q + O(q^{2}) \in S_{k}^{\infty}(\Gamma_{0}(N)),\]
and so 
        $$F \mid T_{k}(p) = \Theta^{k-1}(\phi_{p}).$$
That is, 
        $$F \mid U(p) = \Theta^{k-1}(\phi_{p}) - p^{k-1}F \mid V(p).$$
Applying $U(p^{2})$ to both sides and arguing inductively, we obtain for each $m \geq 0$
        $$F \mid U(p^{2m+1}) = \sum_{j=0}^{m} (-1)^{m-j}p^{(k-1)(m-j)}\Theta^{k-1}(\phi_{p}) \mid U(p^{2j}) + (-1)^{m+1}p^{(k-1)(m+1)}F \mid V(p).$$
For any $j \geq 0$ we have $\Theta^{k-1}(\phi_{p}) \mid U(p^{2j}) \equiv 0 \Mod{p^{(k-1)2j}}$. Hence for $m \geq 0$,
        $$F \mid U(p^{2m+1}) \equiv (-1)^{m}p^{(k-1)m}\Theta^{k-1}(\phi_{p}) \Mod{p^{(k-1)(m+1)}}.$$
The result follows by comparing coefficients of $q$ above. 
\end{proof}

\begin{prop} \label{Cong 2}
For each space $S_{k}(\Gamma_{0}(N))$ listed in Proposition \ref{prop:finitelymany}, if $p$ is an odd, inert prime in the field of complex multiplication for the corresponding newform $g$, then $p \nmid C(p)$. 
\end{prop}

\begin{proof}
When $N \in \{9, 32\}$ we refer the reader to \cite[Lemma 3.3, Lemma 4.4]{Ahlgren-Samart}. For $N \in \{27, 36, 49\}$ we argue analogously. For each such $N$, suppose $p \mid C(p)$. Proposition \ref{prop:1}(c) asserts that 
        $$\Theta(\phi_{p}) = F \mid T_{2}(p) = pF_{p} + C(p)g \equiv 0 \Mod{p}.$$
Thus we have
        $$\phi_{p} \equiv q^{-p} + \sum_{n=1}^{\infty} A_{p}(np)q^{np} \Mod{p}.$$
We now consider each case separately.

First let $N = 27$. By using Sage (for example) to inspect a basis for $M_{2}(\Gamma_{0}(27)),$ one can see that there exists an element
	$$f(z) := \frac{\eta^6(27z)}{\eta^{2}(9z)} = q^{6} + 2q^{15} + O(q^{24}) \in M_{2}(\Gamma_{0}(27)).$$
Then $f^{p} \in M_{2p}(\Gamma_{0}(27))$ has the form
	$$f^{p} \equiv \sum_{n=6}^{\infty} B_{p}(np)q^{np} \equiv q^{6p} + \cdots \Mod{p}.$$
As $\phi_{p} \in M_{0}^{\infty}(\Gamma_{0}(27))$, we find that $h_{p}:=\phi_{p}f^{p} \in M_{2p}(\Gamma_{0}(27))$ has the form
	$$h_{p} \equiv \sum_{n=5}^{\infty} D_{p}(np)q^{np} \equiv q^{5p} + \cdots \Mod{p}.$$
Hence
	$$h_{p} \equiv h_{p} \mid U(p) \mid V(p) \Mod{p}.$$
By Proposition \ref{filt}, we get
	$$w_{p}(h_{p}) = p w_{p}(h_{p} \mid U(p)).$$
As $w_{p}(h_{p}) \equiv 2p \Mod{p-1}$ by Proposition \ref{filt} and $p \mid w_{p}(h_{p})$, we have that $w_{p}(h_{p}) = 2p$, so that $w_{p}(h_{p} \mid U(p)) = 2$. So there exists $h_{0} \in M_{2}^{(p)}(\Gamma_{0}(27))$ such that
	$$h_{0} \equiv h_{p} \mid U(p) = q^{5} + O(q^{6}) \Mod{p}.$$
But by examining a basis for the six-dimensional space $M_{2}(\Gamma_{0}(27))$, we find that there is no such $h_{0}$ of this form. This completes the proof of the claim for $N=27.$ 

Now let $N = 36$. Following a similar argument as above, define
        \[f(z) = q^{12} - 2q^{18} + 3q^{24} + O(q^{30}) \in M_{2}(\Gamma_{0}(36)).\]
Then $f^{p} \in M_{2p}(\Gamma_{0}(36))$ has the form
        \[f^{p} \equiv \sum_{n = 12}^{\infty} B_{p}(np)q^{np} \equiv q^{12p} + \cdots \Mod{p}.\]
Since $\phi_{p} \in M_{0}^{\infty}(\Gamma_{0}(36))$, we have that $h_{p} := \phi_{p}f^{p} \in M_{2p}(\Gamma_{0}(36))$ and has the form 
        \[h_{p} \equiv \sum_{n = 11}^{\infty} D_{p}(np)q^{np} \equiv q^{11p} + \cdots \Mod{p}.\]
Then $h_{p} \mid U(p) \mid V(p) \equiv h_{p} \Mod{p}$. By Proposition \ref{filt}, we get
        \[w_{p}(h_{p}) = pw_{p}(h_{p}\mid U(p)).\]
Since $w_{p}(h_{p}) \equiv 2p \Mod{p-1}$ and $p \mid w_{p}(h_{p})$, we have $w_{p}(h_{p}) = 2p$. Hence $w_{p}(h_{p} \mid U(p)) = 2$. So there exists $h_{0} \in M_{2}^{(p)}(\Gamma_{0}(36))$ such that 
        $$h_{0} \equiv h_{p} \mid U(p) = q^{11} + O(q^{12}) \Mod{p}.$$
But examining a basis of $M_{2}(\Gamma_{0}(36))$, we find that no such form $h_{0}$ can exist. This completes the proof of the claim for $N=36.$ 

Let $N = 49$. If $p = 3$ then $3 \nmid C(3) = -1$ (see the last paragraph on filtrations in \S \ref{Background}; namely the condition $p \nmid 6N$). Let $p > 3$ be inert. There exists an element
	$$f(z) = q^{29} + q^{30} - q^{32} + O(q^{36}) \in M_{8}(\Gamma_{0}(49))$$
so that
	$$f^{p} \equiv \sum_{n=29}^{\infty} B_{p}(np)q^{np} \equiv q^{29p} + \cdots \Mod{p}.$$
Then $h_{p} := \phi_{p}f^{p} \in M_{8p}(\Gamma_{0}(49))$ has the form 
	$$h_{p} \equiv \sum_{n=28}^{\infty} D_{p}(pn)q^{pn} \equiv q^{28p} + \cdots \Mod{p}.$$
and so $h_{p} \equiv h_{p} \mid U(p) \mid V(p) \Mod{p}$. As before, Proposition \ref{filt} shows us that
	$$w_{p}(h_{p}) = pw_{p}(h_{p} \mid U(p)).$$
Since $w_{p}(h_{p}) \equiv 8p \Mod{p-1}$ and $p \mid w_{p}(h_{p})$, we have $w_{p}(h_{p}) = 8p - ap(p-1)$ for some integer $a \geq 0$. As $p > 9$, we must have $a = 0$. Hence $w_{p}(h_{p}) = 8p$, so that $w_{p}(h_{p} \mid U(p)) = 8$. Thus there exists $h_{0} \in M_{8}^{(p)}(\Gamma_{0}(49))$ such that
	$$h_{0} \equiv h_{p} \mid U(p) = q^{28} + O(q^{29}) \Mod{p}.$$
But by examining a basis for the $36$-dimensional space $M_{8}(\Gamma_{0}(49))$, we find that no such $h_{0}$ exists. This completes the proof.
\end{proof}

\section{Proof of Theorem \ref{thm:main1}} \label{sec:main1}

Equation \eqref{Thm1.a} of Theorem \ref{thm:main1} follows from Propositions \ref{Cong 1} and \ref{Cong 2}. To prove equation \eqref{Thm1.b}, note that from part (c) of Proposition \ref{prop:1} it follows that
    \begin{align} \label{fin}
	\frac{F \mid U(p^{2m+1})}{C(p^{2m+1})} - g = \frac{1}{C(p^{2m+1})} \left( p^{(k-1)(2m+1)}F_{p^{2m+1}} - \sum_{j=1}^{2m+1} p^{(k-1)j}F \mid U(p^{2m+1-j}) \mid V(p^{j}) \right).
	\end{align}
Observe
    $$F \mid T(p^{2m}) = \sum_{j=1}^{2m+1} p^{(k-1)(j-1)}F \mid U(p^{2m+1-j}) \mid V(p^{j-1})$$
and on the other hand
    $$F \mid T(p^{2m}) = p^{(k-1)2m}F_{p^{2m}} + C(p^{2m})g = p^{(k-1)2m}F_{p^{2m}}$$
by part (c) of Proposition \ref{prop:1} together with the fact that $C(p^{2m})=0$ since $p$ is inert. Therefore, (\ref{fin}) becomes
	\begin{align} \label{almost} 
	\frac{F \mid U(p^{2m+1})}{C(p^{2m+1})} - g =  \frac{1}{C(p^{2m+1})} \left( p^{(k-1)(2m+1)}F_{p^{2m+1}} -p^{(k-1)(2m+1)}F_{p^{2m}} \mid V(p) \right).
	\end{align}
By Proposition \ref{Cong 1}, this completes the proof.

\bibliographystyle{alpha}
\bibliography{refs}

\begin{thebibliography}{{LMF}21}

\bibitem[AS16]{Ahlgren-Samart}
Scott Ahlgren and Detchat Samart.
\newblock A note on cusp forms as {$p$}-adic limits.
\newblock {\em J. Number Theory}, 168:360--373, 2016.

\bibitem[BOR08]{Bruinier-Ono-Rhoades}
Jan~H. Bruinier, Ken Ono, and Robert~C. Rhoades.
\newblock Differential operators for harmonic weak {M}aass forms and the
  vanishing of {H}ecke eigenvalues.
\newblock {\em Math. Ann.}, 342(3):673--693, 2008.

\bibitem[CP84]{Cox-Parry}
David~A. Cox and Walter~R. Parry.
\newblock Genera of congruence subgroups in {${\bf Q}$}-quaternion algebras.
\newblock {\em J. Reine Angew. Math.}, 351:66--112, 1984.

\bibitem[CP03]{Cummins-Pauli}
C.~J. Cummins and S.~Pauli.
\newblock Congruence subgroups of {${\rm PSL}(2,{\Bbb Z})$} of genus less than
  or equal to 24.
\newblock {\em Experiment. Math.}, 12(2):243--255, 2003.

\bibitem[Dic21]{Dicks}
Robert Dicks.
\newblock Weight 2 cm newforms as p-adic limits, 2021.

\bibitem[DS05]{Diamond-Shurman}
Fred Diamond and Jerry Shurman.
\newblock {\em A first course in modular forms}, volume 228 of {\em Graduate
  Texts in Mathematics}.
\newblock Springer-Verlag, New York, 2005.

\bibitem[EGO10]{El-Guindy-Ono}
Ahmad El-Guindy and Ken Ono.
\newblock Gauss's {${}_2F_1$} hypergeometric function and the congruent number
  elliptic curve.
\newblock {\em Acta Arith.}, 144(3):231--239, 2010.

\bibitem[GKO10]{Guerzhoy-Kent-Ono}
Pavel Guerzhoy, Zachary~A. Kent, and Ken Ono.
\newblock {$p$}-adic coupling of mock modular forms and shadows.
\newblock {\em Proc. Natl. Acad. Sci. USA}, 107(14):6169--6174, 2010.

\bibitem[JM19]{JenkinsMolnar19}
Paul Jenkins and Grant Molnar.
\newblock Zagier duality for level {$p$} weakly holomorphic modular forms.
\newblock {\em Ramanujan J.}, 50(1):93--109, 2019.

\bibitem[Joc82]{Jochnowitz}
Naomi Jochnowitz.
\newblock Congruences between systems of eigenvalues of modular forms.
\newblock {\em Trans. Amer. Math. Soc.}, 270(1):269--285, 1982.

\bibitem[{LMF}21]{LMFDB}
The {LMFDB Collaboration}.
\newblock The {L}-functions and modular forms database.
\newblock \url{http://www.lmfdb.org}, 2021.
\newblock [Online; accessed 21 January 2021].

\bibitem[Ono04]{Ono}
Ken Ono.
\newblock {\em The web of modularity: arithmetic of the coefficients of modular
  forms and {$q$}-series}, volume 102 of {\em CBMS Regional Conference Series
  in Mathematics}.
\newblock Published for the Conference Board of the Mathematical Sciences,
  Washington, DC; by the American Mathematical Society, Providence, RI, 2004.

\end{thebibliography}

\end{document}